\documentclass{amsart}

\usepackage{setspace}
\usepackage{amsmath,amsthm}
\usepackage{amsfonts,amssymb}
\usepackage{accents}
\usepackage{enumerate}
\usepackage{color}
\usepackage{graphicx}

\usepackage{mathrsfs}
\usepackage{mathtools 
}

\usepackage[
pdftex,hyperfootnotes]{hyperref}

\usepackage{nicematrix}

\usepackage[ansinew]{inputenc}
\usepackage[T1]{fontenc}

\usepackage{textcomp}
\usepackage{charter}
\usepackage[]{bera}
\usepackage[charter]{mathdesign}

\usepackage{longtable}

\usepackage{enumerate}

\allowdisplaybreaks

\newcommand{\lvt}{\left|\kern-1.35pt\left|\kern-1.3pt\left|}
\newcommand{\rvt}{\right|\kern-1.3pt\right|\kern-1.35pt\right|}

\newtheorem{thm}{Theorem}[section]

\theoremstyle{remark}

 \def\d{\mathrm{d}}

 \def\l{{\lambda}}

\def\lla{\langle{\kern-2.5pt}\langle} 
\def\rra{\rangle{\kern-2.5pt}\rangle}

\graphicspath{{./}}

\begin{document}
 
\title{
Banded matrices and their orthogonality}

\author[Branquinho]{Amílcar Branquinho}
\address{CMUC, Department of Mathematics, University of Coimbra, 3000-143 Coimbra, Portugal}
\email{ajplb@mat.uc.pt}
\thanks{The first author thanks
Center for Mathematics of the University of Coim\-bra UIDB/00324/2020 (funded by the Portuguese Government through FCT/MCTES)}

\author[Foulquié-Moreno]{Ana Foulquié-Moreno}
\address{CIDMA, Departamento de Matemática, Universidade de Aveiro, 3810-193 Aveiro, Portugal}
\email{foulquie@ua.pt}
\thanks{The second  author acknowledges
Center for Research \& Development in Mathematics and Applications 
is supported through the Portuguese Foundation for Science and Technology
references UIDB/04106/2020 and UIDP/04106/2020}

\author[Mañas]{Manuel Mañas}
\address{Departamento de Física Teórica, Universidad Complutense de Madrid, Plaza Ciencias 1,
\linebreak
28040-Madrid, Spain}
\email{manuel.manas@ucm.es}
\thanks{The third author was partially supported by the
Spanish ``Agencia Estatal de Investigación'' research
 pro\-jects:
[PGC2018-096504-B-C33], \emph{Ortogonalidad y~Apro\-ximación: Teoría y Aplicaciones en Física Mate\-má\-tica},
and
[PID2021- 122154NB-I00], \emph{Ortogonalidad y aproximación con aplicaciones en machine learning y teoría de la probabilidad}}


\date{\today}


 
\begin{abstract}
Banded bounded matrices, which represent non normal operators, of oscillatory type that admit a positive bidiagonal factorization are considered.
To motivate the relevance of the oscillatory character the Favard theorem for Jacobi matrices is revisited and it is shown that after an adequate shift of the Jacobi matrix one gets an oscillatory matrix.
In this work we present a spectral theorem for this type of operators and show how the theory of multiple orthogonal polynomials apply.
\end{abstract} 
 
\maketitle

%

\section{Introduction} \label{sec:1}

Since the celebrated Stone theorem it is known that a selfadjoint operator with simple spectrum defined in a Hilbert space can be represented as a symmetric Jacobi operator in a suitable orthonormal basis,
cf. for instance~\cite{Akhiezer_Glazman}.
The infinite matrix representative of these Jacobi operators is tridiagonal, which leads to a three term recurrence relation definition of an orthonormal polynomial basis, with respect to the spectral measure of the Jacobi operator.
In this way the theory of orthogonal polynomials is instrumental to understand the spectrality of a selfadjoint operator in a Hilbert~space.

The construction of the spectral measure for these Jacobi operators can be reached by means of a Markov theorem, i.e. as a uniform limit of ratio for polynomials satisfying the three term recurrence relations coming from the Jacobi operator.
The spectral measure is in fact a measure of orthogonality for the sequence of orthogonal polynomials related to the Jacobi operator, as they share its moments (cf.~\cite{Chihara,Ismail}).

In the banded Hessenberg operators, the selfadjointness no more takes place (cf. for instance~\cite{PBF_1}). Nevertheless these banded operators case can be read as high order (bigger than $3$ term) recurrence relations for important recurrence polynomials called type II ones and linear forms of type I in the fundamental book of Nikishin and Sorokin~\cite{nikishin_sorokin}.
Moreover, a biorthogonality can be derived between these two systems (the ones of polynomials and the linear forms) that stands as duality relations. In fact, this is a sort of Favard theorem for banded Hessenberg operators.
Several authors try to extend these Favard type theorems (cf.~\cite{Sorokin_Van_Iseghem_1,Sorokin_Van_Iseghem_2,Sorokin_Van_Iseghem_3}) to general banded operators (with more than one upper diagonals). In this case the multiple orthogonality has been extended to the so called mixed one.
A natural question arise from these facts: Do multiple (or mixed) orthogonal polynomials
can describe banded operators?
That is, they can be used to give an interpretations for the spectral points, as well as of the spectral measure of the operator.

A first attempt, in the Hessenberg context, has been made by Valery Kalygin~\cite{Kalyagin,Kaliaguine} where the author defines a class of operators related to the Hermite--Padé approximants and connects their spectral analysis with the asymptotic properties of polynomials defined by systems of orthogonality relations (that coincides with the notion of multiple orthogonality).
Some years later, the author in a joint work with Alexander Aptekarev and Janette Van Iseghem~\cite{Aptekarev_Kaliaguine_VanIseghem} made the analysis of banded Hessenberg operators with one upper diagonal based on the analysis of the genetic sums formulas for the moments of the operator.

Since the magnificent book of Gantmacher and Krein~\cite{Gantmacher-Krein} it is well known that the feature of being oscillatory of a Jacobi matrix is important in the study of the spectral points of a Jacobi operator. They are also instrumental in the study of the orthogonal polynomials sequences used to define the spectral measure of the operator.
In this work we highlight, avoiding technicalities, the main results in~\cite{espectral} about the spectral measure associated with a bounded banded operators with positive bidiagonal factorization (in this case the operator is oscillatory). 
We take the Jacobi operators as case study in \S~\ref{sec:2}, to give an explanation of the ideas we intend to use in this paper. 
To a good introduction in the subject of oscillatory matrices we refer the book by Fallat and Johnson~\cite{Fallat-Johnson}.
In~\S~\ref{sec:3} we consider an oscillatory banded operator with $2$ upperdiagonals and $3$ subdiagonals as a representative case, and define the left and right eigenvectors for its truncation matrices.
We succeed in giving an interpretation of the discrete mixed multiple orthogonality between the left and right eigenvectors of the oscillatory banded matrix. The last section is devoted to the spectral theorem for the banded operator, where we also present a Gauss like quadrature formula.


\section{Revisiting bounded Jacobi matrices} \label{sec:2}

To illustrate the approach of this paper to banded matrices we will consider first the classical spectral theory for Jacobi matrices. We will proceed emphasizing the ideas of our approach that will be extended in the next section to banded matrices.

\subsection{The Jacobi matrix and recursion polynomials}

Let us consider the tridiagonal semi-infinite real matrix
\begin{align}\label{eq:Jacobi_matrix_def}
J & \coloneqq
\left[\begin{NiceMatrix}[columns-width = auto]
 m_0& 1 &0&\Cdots&\\[-.25cm]
{\ell}_1 &m_1& 1& \Ddots&\\
0&\ell_2&m_2&1&\\
\Vdots& \Ddots& \Ddots & \Ddots&\Ddots\\
&&&&\\
\end{NiceMatrix}
\right],
&& \ell_j >0 , && j \in \mathbb Z_+ \coloneqq \{ 1,2, \ldots \} .
\end{align}
We assume for our propose that $\ell_{0} = 1$ as well as $\mathbb N \coloneqq \{ 0,1, \ldots \} $.
This matrix is symmetrizable, by the positive diagonal matrix
\begin{align*}
H&=\operatorname{diag} \left[ \begin{NiceMatrix}
H_0 & H_1 & \Cdots \end{NiceMatrix} \right], & H_0&=1, &H_n\coloneqq \ell_1\cdots \ell_{n}.
\end{align*}
In fact, $H^{-\frac{1}{2}}JH^{\frac{1}{2}}$ is symmetric.
The recursion polynomials $\big\{ P_n \big\}_{n \in \mathbb N}$ are monic polynomials with $\deg P_n= n$ satisfying the recursion relation
\begin{align}\label{eq:recursion_Jacobi}
\ell_n P_{n-1} +m_n P_n+P_{n+1}&=xP_n, & n&\in \mathbb N,
\end{align}
with $P_0=1$ and $P_{-1} = 0$. That is, 
\begin{align*}
P(x) =\left[\begin{NiceMatrix}
P_0 (x) &
P_1 (x) &
\Cdots
\end{NiceMatrix}\right]^\top ,
\end{align*}
satisfies the eigenvalue property $J P = x P $. 
Dual to $P(x)$ we introduce 
\begin{align*}
Q (x) &=
\left[
\begin{matrix}
Q_0(x) & Q_1(x)&\cdots
\end{matrix}
\right], 
\end{align*}
that are left eigenvectors of the semi-infinite matrix $J$, i.e.,
$Q(x) J=x Q(x)$, and the initial conditions, that determine these polynomials uniquely, are taken as $Q_0=1$.
It can be easily shown that
$ Q_n=\dfrac{P_n}{H_n} $.
Let us introduce the leading principal submatrices $J^{[N]}=J[\{0,1,\ldots,N\}]$,~i.e.
\begin{align*}
J^{[N]}&\coloneqq
\left[\begin{NiceMatrix}[columns-width = auto]
m_0& 1 &0&\Cdots&&0 \\[-.25cm]
\ell_1 &m_1& 1& \Ddots&&\Vdots\\
0&\ell_2&m_2&1&&\\
\Vdots& \Ddots& \Ddots & \Ddots&\Ddots&0\\
&&&&&1\\
0&\Cdots&&0&\ell_N&m_N
\end{NiceMatrix}\right]\in\mathbb R^{(N+1)\times (N+1)},
\end{align*}
with $ \Delta_N \coloneqq \det J^{[N]}$.
As $J^{[N]}$ is symmetrizable its eigenvalues are real.
Then, from the given definition of the recursion polynomials we get that they are the characteristic polynomials of these submatrices
\begin{align*}
P_{N+1}(x)=\det\big( xI_{N+1}-J^{[N]}\big), && N \in \mathbb N .
\end{align*}
To prove it we expand the determinant along the last row to see that these characteristic polynomials satisfy \eqref{eq:recursion_Jacobi} and have the same initial condition. 


The next sequence of polynomials will play an important role in the study of the spectrality of the Jacobi operator.
The second kind polynomials $\big\{P_n^{(1)}\big\}_{n \in \mathbb N}$ are defined by the following initial conditions
\begin{align}\label{eq:initial_conditions_second}
P^{(1)}_{-1}&=1, &P^{(1)}_0&=0,
\end{align}
For the recursion polynomials of the second kind one finds
\begin{align*}
P^{(1)}_{N+1} & = e_1^\top \operatorname{adj}\big(x I_{N+1}-J^{[N]}\big)e_1 \\
 & =
\left|
\begin{NiceMatrix}[columns-width = auto]
x-m_{1}& -1 &0&\Cdots&&0\\[-.25cm]
-\ell_{2} & x-m_{2}& -1& \Ddots&&\Vdots\\
0& -\ell_{3}& x-m_{3}&-1&&\\
\Vdots& \Ddots& \Ddots & \Ddots&\Ddots&0\\
&&&&&1\\
0&\Cdots&&0&-\ell_N&x-m_N
\end{NiceMatrix}\right| .
\end{align*}
Here $e_1$ is the semi-infinite vector with entries $\delta_{1,n}$ and by $ \operatorname{adj}$ we mean the adjugate matrix.
This can be proven by expanding the determinant along the last row to get that the recursion relation \eqref{eq:recursion_Jacobi} as well as the initial conditions \eqref{eq:initial_conditions_second} are fulfilled.
Associated with the principal submatrices 
\begin{align*}
J^{[N,k]}&\coloneqq \left[\begin{NiceMatrix}[columns-width = auto]
m_{k}& 1 &0&\Cdots&&0\\[-.25cm]
\ell_{k+1} &m_{k+1}& 1& \Ddots&&\Vdots\\
0&\ell_{k+2}&m_{k+2}&1&&\\
\Vdots& \Ddots& \Ddots & \Ddots&\Ddots&0\\
&&&&&1\\
0&\Cdots&&0&\ell_N&m_N
\end{NiceMatrix}
\right]\in\mathbb R^{(N+1-k)\times (N+1-k)} ,
\end{align*}
with $ \Delta_{N,k} \coloneqq \det J^{[N,k]}$, we define the truncated polynomials
\begin{align*}
P^{[k]}_{N+1}= \det (x I_{N+1-k}-J^{[N,k]}), && k \in \{0,1,\ldots,N\} .
\end{align*}
An expansion of these determinants along the first row show that they satisfy the recursion relation,
$P^{[N+1]}_{N+1}=1$, $P^{[N+2]}_{N+1}=0$ and
\begin{align*}
\ell_{k+1}P^{[k+2]}_{N+1}+m_kP^{[k+1]}_{N+1}+P^{[k]}_{N+1}&=xP^{[k+1]}_{N+1}, & k&\in\{0,1,\ldots,N\} .
\end{align*}

\subsection{Eigenvalues, left and right generalized eigenvectors}
Despite the fact that, in the Jacobi scenario, as we will show later, the eigenvalues are going to be simple, we allow at this stage of the discussion that possible degenerate eigenvalues could appear. Let us introduce
$P^{\langle N\rangle} (x) \coloneqq 
\left[
\begin{matrix}
P_0 (x) &
P_1 (x) &
\cdots &
P_N (x)
\end{matrix}
\right]^\top$
so that
\begin{align*}
J^{[N]} P^{\langle N\rangle}(x)
+
\left[\begin{matrix}
0 & \cdots & 0 & P_{N+1} (x)
\end{matrix}
\right]^\top = x P^{\langle N\rangle}(x) .
\end{align*}
From this equation and the relations obtained from it by taking derivatives we get that the set of zeros and corresponding multiplicities of the polynomial $P_{N+1}(x)$, say $\sigma^{[N]}=\big\{ \lambda^{[N]}_n,\kappa_n^{[N]} : n= 1, \ldots , r^{[N]} \big\}$,
$\kappa_1^{[N]} + \cdots+\kappa^{[N]}_{r^{[N]}}=N+1$ is the spectrum of $ J^{[N]}$, with corresponding right generalized eigenvectors
\begin{align*}
u^{\langle N\rangle,(j)}_n&\coloneqq 
\frac{\operatorname d^k P^{\langle N\rangle}}{\operatorname d x^j}
\bigg|_{x=\lambda^{[N]}_{n}}, &n&\in\{1,\ldots,r^{[N]}\},& j&\in\{0,\ldots,\kappa^{[N]}_n-1\},
\end{align*}
with rank $r^{[N]}$. The~vectors $u^{\langle N\rangle,(0)}_n$ are right eigenvectors, and the set 
\begin{align*}
\big\{ u^{\langle N\rangle,(\kappa^{[N]}_n-1)}_n,\ldots, u^{\langle N\rangle,(1)}_n,u^{\langle N\rangle,(0)}_n\big\} ,
\end{align*}
 is a right Jordan chain.
Similarly, we construct a set of generalized left eigenvectors of the matrix~$J^{[N]}$. For 
$ Q^{\langle N\rangle} (x) \coloneqq
 \left[
 \begin{matrix}
 Q_{0} (x) &Q_{1} (x) &\cdots & Q_{N} (x)
 \end{matrix}
 \right]$
 and $k\in\{1,\ldots,N+1\}$, the~vectors
 \begin{align*}
 w^{\langle N\rangle ,(j)}_n&\coloneqq \frac{\operatorname d^j Q^{\langle N\rangle}}{\operatorname d x^j }\bigg|_{x=\lambda^{[N]}_n}, &n&\in\{1,\ldots,r^{[N]}\},& j&\in\{0,\ldots,\kappa^{[N]}_n-1\},
 \end{align*}
 are generalized left eigenvectors of $ J^{[N]}$. The $w^{\langle N\rangle ,(0)}_n$ are the left eigenvectors and 
 the~set
 \begin{align*}
 \big\{w^{\langle N\rangle,(\kappa_n^{[N]}-1)}_n,\ldots, w^{\langle N\rangle,(1)}_n,w^{\langle N\rangle,(0)}_n\big\} ,
 \end{align*}
 is a left Jordan chain.

\subsection{The Christoffel--Darboux formulas}

Noticing that
\begin{align*}
\left[
\begin{matrix}
Q_{0}(x)& \cdots & Q_{N}(x)
\end{matrix}
\right]
J^{[N]}
+
\left[\begin{matrix}
Q_{N+1} (x)& 0&\cdots & 0
\end{matrix}
\right]
&=x
\left[\begin{matrix}
Q_{0}(x)& \cdots & Q_{N}(x)
\end{matrix}
\right] ,
\\
J^{[N]}
\left[
\begin{matrix}
P_0(y) & \cdots&
P_N(y)
\end{matrix}
\right]^\top
+
\left[
\begin{matrix}
0 & \cdots & 0 & P_{N+1} (y)
\end{matrix}
\right]^\top &
=y 
\left[
\begin{matrix}
P_0(y) & \cdots & P_N(y)
\end{matrix}
\right]^\top ,
\end{align*}
we deduce that
\begin{multline*}
\left[
\begin{matrix}
Q_{N+1} (x)& 0&\cdots &
0
\end{matrix}
\right]
\left[
\begin{matrix}
P_0(y) \\
\vdots \\
P_N(y)
\end{matrix}
\right]
-
\left[
\begin{matrix}
Q_{0}(x)& \cdots & Q_{N}(x)
\end{matrix}
\right]
\left[
\begin{matrix}
0\\[-.15cm]
\vdots \\
0
\\
P_{N+1} (y)
\end{matrix}
\right]
 \\
 =(x-y )
\left[
\begin{matrix}
Q_{0}(x)& \cdots & Q_{N}(x)
\end{matrix}
\right]
\left[
\begin{matrix}
P_0(y)\\
\vdots\\
P_N(y)
\end{matrix}
\right]
,
\end{multline*}
and we find the following Christoffel--Darboux relation 
\begin{align*}
\sum_{n=0}^{N}P_{n} (x) P_n (y)=\frac{1}{H_N}\frac{P_{N}(x)P_{N+1}(y)-P_{N+1}(x)P_N(y)}{x-y} , 
\end{align*}
and the associated confluent Christoffel--Darboux relation 
\begin{align*}
\sum_{n=0}^{N} P_n^2 (x) = \frac{ P^\prime_{N} (x) P_{N+1} (x) -P^\prime_{N+1} (x) P_N (x) }{H_{N}} .
\end{align*}

\subsection{The case of simple eigenvalues and discrete orthogonality}

Let us assume that for each $N$ the polynomial $P_{N+1}$ has simple zeros at
the set $\big\{\lambda^{[N]}_k : k= 1, \ldots , N+1\big\}
 $.
Then, we have right eigenvectors $u^{\langle N\rangle}_n=u_n^{\langle N\rangle,(0)}$, and left eigenvectors $w_n^{\langle N\rangle}=w_n^{\langle N\rangle,(0)}$, instead of generalized eigenvectors, and the Jordan chains have only one element, the corresponding eigenvectors. Later on, we will see that, in fact, this is the case for the Jacobi matrix given in \eqref{eq:Jacobi_matrix_def} with $\ell_k>0$, $k \in \mathbb Z_+$.
 For $k=1,\ldots,N+1$, the vectors $u_k^{\langle N\rangle}\coloneqq P^{\langle N\rangle}\big(\lambda^{[N]}_k\big)$ ($\tilde w^{\langle N\rangle}_k\coloneqq Q_k^{\langle N\rangle}$) are right (left) eigenvectors of $ J^{[N]}$.
Given any set of left eigenvectors $\big\{\tilde w^{[N]}_k : k=1 , \ldots , N+1 \big\}$, a normalized basis $\big\{w^{\langle N\rangle}_k : k= 1, \ldots , N+1 \big\}$ biorthogonal to the basis of right eigenvectors $\big\{u^{\langle N\rangle}_k : k= 1, \ldots , N+1 \big\}
 $ is given by
\begin{align*}
w^{\langle N\rangle}_{k,n}=\frac{\tilde w^{\langle N\rangle}_{k,n}}{\sum_{l=1}^{N+1}\tilde w^{\langle N\rangle}_{k,l}P_{l-1}(\lambda^{[N]}_k)}.
\end{align*}
In particular, the following expression holds
\begin{align*}
w^{\langle N\rangle}_{k,n}&=\frac{Q_{n-1}(\lambda^{[N]}_k) }{\sum_{l=0}^{N}Q_{l}(\lambda^{[N]}_k)P_{l}(\lambda^{[N]}_k)}=-
H_{N}
\frac{Q_{n-1}(\lambda^{[N]}_k)
}{
P_N (\lambda^{[N]}_k)P^\prime_{N+1}(\lambda^{[N]}_k)} .
\end{align*}
Moreover,
we can write 
\begin{align*}
w^{[N]}_{k,n}&= Q_{n-1}(\lambda^{[N]}_k )\mu^{[N]}_{k}, &
\mu_{k}^{[N]}&\coloneqq w^{\langle N\rangle}_{k,1}=\frac{1}{\sum_{l=0}^{N}Q_{l} (\lambda^{[N]}_k )P_{l}(\lambda^{[N]}_k )}.
\end{align*}
Here $\mu^{[N]}_k$ will be call the mass at that eigenvalue.
Observe that for these masses we~have 
\begin{align}\label{eq:masses_Christoffel_numbers}
\mu_{k}^{[N]}&=\frac{1}{\lambdaup_{N,k}}, &\lambdaup_{N,k}\coloneqq \sum_{l=0}^{N}\frac{P^2_{l} (\lambda^{[N]}_k)}{H_l} ,
\end{align}
where $\lambdaup_{N,k}$ are the so called Christoffel numbers. This leads to the positivity of the masses. For more on Christoffel numbers see~\cite[\S 2.4]{Ismail}.
Hence, the masses 
can be expressed as
\begin{align}\label{eq:masses_Jacobi_intro}
\mu_{k}^{[N]}&=\frac{P_{N+1}^{(1)} (\lambda^{[N]}_k)}{P^\prime_{N+1}(\lambda^{[N]}_k)}.
\end{align}
As the masses are positive we conclude that recursion polynomials $P_{N+1}$ strictly interlace its second type polynomials $P_{N+1}^{(1)}$.

 For the corresponding matrices $U$ (with columns the right eigenvectors~$u_k$ arranged in the standard order) and~$W$ and (with rows the left eigenvectors~$w_k$ arranged in the standard order) we find 
$UW=WU=I_{N+1}$ and, in terms of $D=\operatorname{diag}
\big[ \begin{matrix} \lambda^{[N]}_1,\ldots,\lambda^{[N]}_{N+1} \end{matrix}\big]$, we have
$UD^nW=\big(J^{[N]}\big)^n$. Consequently, the following equations are satisfied
\begin{align}
\label{eq:Pw}
\sum_{j=1}^{N+1}P_{k-1}(\lambda^{[N]}_j)w^{[N]}_{j,l}&=\delta_{k,l}, && k,l \in \{ 1, \ldots , N+1 \}, \\
\label{eq:B_lambda_n_w}
\sum_{j=1}^{N+1}P_{k-1}(\lambda^{[N]}_j)(\lambda^{[N]}_j)^nw^{[N]}_{j,l}&=\big(( J^{[N]})^n\big)_{k,l},
 && n \in \{ 0, \ldots , N \} .
\end{align}
In terms of the singular measure, with support on the zeros of $P_{N+1}$, given~by
\begin{align*}
\mu^{[N]}&\coloneqq \sum_{j=1}^{N+1}\mu^{[N]}_{j}\delta\big(z-\lambda^{[N]}_j\big), 
\end{align*}
 conditions \eqref{eq:Pw} and \eqref{eq:B_lambda_n_w} can be recast as the 
 orthogonality relations
\begin{align*}
\left\langle\mu^{[N]}, x^k P_{n}\right\rangle&=0, & k&=0,\ldots,n-1.
\end{align*}
Similarly, we obtain 
\begin{align*}
\left\langle Q_{k}\mu^{[N]} ,x^n\right\rangle&=0, & n&\in\{0,1,\ldots,k-1\}, & k\in\{1,\ldots, N\}.
\end{align*}
Finally, the following biorthogonal relations are satisfied
\begin{align*}
\left\langle Q_{k}\mu^{[N]},P_l\right\rangle=\delta_{k,l}, && k,l\in\{0,\ldots,N\} .
\end{align*}
\enlargethispage{.245cm}
This way of writing the orthogonality relations is prepared to be extended to multiple orthogonality later in the paper. Notice also that the polynomials $p_n\coloneqq \frac{P_n}{\sqrt H_n}=\sqrt H_n Q_n$ are orthonormal polynomials.
From $ \sum_{j=1}^{N+1}w^{\langle N\rangle}_{j,1}=1$ we see that
$\sum_{j=1}^{N+1}\mu^{[N]}_{j,1}=1$. Then, we find a Lebesgue--Stieltjes representation of this singular measure.
In terms of the piecewise continuous function
\begin{align*}
\psi^{[N]}&\coloneqq \begin{cases}
0, & x<\lambda^{[N]}_{N+1},\\[2pt]
\mu^{[N]}_{1}+\cdots+\mu^{[N]}_{k}, & \lambda^{[N]}_{k+1}\leqslant x< \lambda^{[N]}_{k}, \quad k\in\{1,\ldots,N\},\\[2pt]
\mu^{[N]}_{1}+\cdots+\mu^{[N]}_{N+1}=1 , &x\geqslant \lambda_{1}^{[N]},
\end{cases}
\end{align*}
we have $\mu ^{[N]}=\operatorname d\psi^{[N]}$.
The resolvent matrix $ R^{[N]} $ of the leading principal submatrix~$ J^{[N]}$~is 
\begin{align*}
R^{[N]} (z) \coloneqq \big(z I_{N+1}- J^{[N]}\big)^{-1} = \frac{\operatorname{adj}\big(z I_{N+1}- J^{[N]}\big)}{\det(z I_{N+1}-J^{[N]})}.
\end{align*}
From the spectral decomposition of the matrix $ J^{[N]}$, we obtain 
$R^{[N]} (z) = U(zI_{N+1}-D)^{-1}W$. The corresponding Weyl function $S^{[N]}$ defined by
\begin{align*}
S^{[N]} (z) &\coloneqq e_1^\top\big(zI_{N+1}- T^{[N]}\big)^{-1} e_1,
\end{align*}
can be expressed as follows
\begin{align*}
S^{[N]}(z)&=\frac{P^{(1)}_{N+1}(z)}{P_{N+1}(z)}=\sum_{n=1}^{N+1}\frac{\mu^{[N]}_{n}}{z-\lambda^{[N]}_n}.
\end{align*}

\subsection{Oscillatory matrices, interlacing properties and spectral theorem}\label{section:intro_spectral}

If the matrix $J$ is bounded all the possible eigenvalues of the submatrices~$J^{[N]}$ belong to the disk $D(0,\|J\|)$.
As all the eigenvalues are real, let us consider those that are negative, and let $b$ be the supreme of the absolute values of all negative eigenvalues. Notice that~$b \leqslant \|J\|$. 
\begin{thm}\cite[Theorem 1.7]{espectral}
For $s\geqslant b$ the matrix $J_s=J+s \, I$ is oscillatory.
\end{thm}


As a corollary of this theorem, we get that all eigenvalues are simple, and that~$P_{N+1}$ interlaces $P_{N}$ and $P^{(1)}_{N+1}$.
 Indeed, the characteristic polynomial $P_{N+1}(x-s)$ of the oscillatory matrix $J^{[N]}_s$ interlaces the characteristic polynomials of the submatrices $J^{[N]}_s(1)=J^{[N,1]}_s$, i.e. $P^{(1)}_{N+1}(x-s)$, and of $J^{[N]}_s(N+1)=J^{[N-1]}_s$, i.e. $P_N(x-s)$. 
Hence, from \eqref{eq:masses_Jacobi_intro} we may deduce the positivity of the masses, i.e. $\mu_k^{[N]}>0$ for all $k\in\{1,\ldots,N+1\}$ from the oscillatory character $J^{[N]}_s$, appealing to \eqref{eq:masses_Jacobi_intro} and the interlacing property of the recursion polynomials and its polynomial of the second type. Hence, we have at hand an alternative way to \eqref{eq:masses_Christoffel_numbers}, in terms of the Christoffel numbers, to find positivity.

Now, recognizing the fact that for the Jacobi matrix $J$ as in \eqref{eq:Jacobi_matrix_def} exists an~$s$ such that $J_s=sI+J$ is oscillatory gives a very fast access to important properties of the associated recursion polynomials and to the Stone--Shoat--Favard's spectral theorem, see~\cite[\S 4.1]{Simon}.

According to Helly's selection Principle~\cite{Chihara} for any uniformly bounded sequence $\big\{\psi^{[N]}\big\}_{N \in \mathbb N}$ of non-decreasing functions defined in $\mathbb R$, there exists a convergent subsequence converging to a non-decreasing function $\psi$ defined in $\mathbb R$. This leads to Helly's second theorem. Any uniformly bounded sequence $\big\{\psi^{[N]}\big\}_{N \in \mathbb N}$ of non-decreasing functions on a compact interval $[a,b]$ 
with limit function $\psi$, then for any continuous function $f$ in $[a,b]$ we have
\begin{align*}
\lim_{N\to\infty} \int_a^b f(x)\operatorname d\psi^{[N]}(x)=\int_a^bf \operatorname d \psi(x).
\end{align*}
Therefore, the positivity of the masses ensures that the functions $\psi^{[N]}$ are non decreasing and uniformly bounded by unity. Hence, Helly's results lead to the existence of a nondecreasing functions $\psi$ and corresponding positive Lebesgue--Stieltjes measures $\operatorname d\psi$ with compact support $\Delta = [a,b]$ such that the orthogonal relations hold
\begin{align*}
\int_\Delta x^k P_{n}(x)\operatorname d \psi(x)&=0, & k&\in\{0,\ldots,n-1\}, && \text{``type II''} \\
\int_{\Delta} Q_{k}(x)\operatorname d \psi(x)x^n&=0, & n&\in\{0,1,\ldots,k-1\}. && \text{``type I''} 
\end{align*}
These polynomial sequences of types II and I are biorthogonal, i.e.,
\begin{align*}
\int_{\Delta} Q_{k}(x)\operatorname d \psi(x)P_l(x)&=\delta_{k,l},
&&
k,l\in \mathbb N .
\end{align*}
Also the type I and type II are just the same thing in this standard non multiple situation. 
Helly's second theorem leads to the spectral representation in terms of the spectral function $\psi$
\begin{align*}
e_1^\top J^k e_1 =\int_\Delta t^k\operatorname d\psi(t), && n \in \mathbb N,
 &&
e_1^\top (z I- J)^{-1}e_1
 =\int_\Delta\frac{\operatorname d\psi(t)}{z-t}.
\end{align*}

\section{Mixed Multiple orthogonality} \label{sec:3}

In this paper we consider a bounded banded operator $T$ whose semi-infinite matrix representation
has $2$ superdiagonals and $3$ subdiagonals, given by
\begin{align}\label{eq:monic_Hessenberg}
T&=
\left[\begin{matrix}
T_{0,0} &T_{0,1} &T_{0,2} \\[.1cm]
T_{1,0}& T_{1,1}& T_{1,2} & T_{1,3} \\[.1cm]
T_{2,0}&T_{2,1}& T_{2,2}& T_{2,3} & T_{2,4} \\[.1cm]
T_{3,0}&T_{3,1}& T_{3,2}& T_{3,3} & T_{3,4} & T_{3,5} \\[.1cm]
 &T_{4,1}&T_{4,2}& T_{4,3}& T_{4,4} & T_{4,5} & T_{4,6} \\
 && \ddots&\ddots& \ddots&\ddots& \ddots & \ddots
\end{matrix}
\right],
&&
T_{n+3,n} \times T_{n,n+2} \neq 0, &&
n\in\mathbb N,
\end{align}
with leading principal submatrices $T^{[N]}=T[\{0,1,\ldots,N\}]$.
In what follows we will show that when this matrix, after a shift, admits a positive bidiagonal factorization and, consequently, is oscillatory, we can find a spectral Favard theorem. Now, we have mixed multiple orthogonal polynomial sequences with respect to a matrix of positive
Lebes\-gue--Stieltjes measures.

\subsection{Matrix interpretation of the multiple orthogonality}


Let us consider a matrix of functions,~$\Psi$, which are right continuous and of bounded variation in a closed interval,~$\Delta$, as well as the associated matrix of Lebesgue--Stieltjes measures, $\d \Psi$, i.e.
\begin{align*}
\Psi=
\left[\begin{matrix}
\psi_{1,1}&\psi_{1,2} & \psi_{1,3} \\
\psi_{2,1}&\psi_{2,2} & \psi_{2,3}
\end{matrix}
\right] && \text{and} &&
\d\Psi=
\left[\begin{matrix}
\d\psi_{1,1}& \d\psi_{1,2} & \d\psi_{1,3} \\
\d\psi_{2,1}& \d\psi_{2,2} & \d\psi_{2,3}
\end{matrix}
\right].
\end{align*}
The matrix of moments is given by
\begin{align} \label{eq:matriz_momentos}
\mathcal M
= \left[
\begin{matrix}
\Psi_{k+j}
\end{matrix}
\right]_{j=0,1,\ldots}^{k=0,1,\ldots} 
&&
\text{where} && 
\Psi_j = \int_\Delta x^{j} \, \d \Psi (x)
, &&
j \in \mathbb N .
\end{align}
Whenever 
 the leading principal submatrices of $\mathcal M$, $\mathcal M_n$, are nonsingular,~i.e.
\begin{align} \label{eq:regular}
\det \mathcal M_n \not = 0 , && n \in \mathbb N;
\end{align}
there exists sequences of vector polynomials
\begin{align*}
\mathscr B_n = \left[ \begin{matrix} B_n^{1} & B_n^{2}\end{matrix} \right],
&&
\mathscr A_n = \left[ \begin{matrix} A_n^{1} & A_n^{2} & A_n^{3} \end{matrix} \right],
&&
n \in \mathbb N ,
\end{align*}
with degrees
\begin{align*}
\deg B^{b}_{n} &=\left\lceil\frac{n+2-b}{2} \right\rceil-1 , &\deg A^{a}_{n} =\left\lceil\frac{n+2-a}{3} \right\rceil-1,
\end{align*}
for $b \in \{ 1,2 \}$ and $a \in \{ 1,2,3 \}$, such that the following biorthogonality conditions takes~place
\begin{align}\label{eq:biorthogonality}
\int_\Delta \mathscr B_n (x) \d \psi (x) \mathscr A_m^\top (x) = \delta_{n,m}, && n , m \in \mathbb N .
\end{align}
In this case there exists an infinite matrix $T$ of type~\eqref{eq:monic_Hessenberg} such that the following recursion relations hold
\begin{align*}
T \pmb{\operatorname B}^{b} =x \pmb{\operatorname B}^{b} , &&
\pmb{\operatorname A}^{a} T=x \pmb{\operatorname A}^{a} , && b \in \{1,2 \} , && a \in \{ 1,2,3 \} ,
\end{align*}
where
\begin{align*}
\pmb{\operatorname B}^{b} 
=\begin{bNiceMatrix}
B^{b}_0 &B^{b}_1 & \Cdots
\end{bNiceMatrix}^\top, && b \in\{1,2\} , 
&&
\pmb{\operatorname A}^{a} 
=\left[\begin{NiceMatrix}
A^{a}_0 & A^{a}_1& \Cdots
\end{NiceMatrix}\right], && a \in\{1,2,3\} ,
\end{align*}
determined by the initial conditions
\begin{align*}
\begin{cases}
A^{1}_0=1 , \\
A^{1}_1= \nu^{1}_1 , \\
A^{1}_{2}=\nu^{1}_{2} ,
\end{cases}
&&
\begin{cases}
A^{2}_0=0 , \\
A^{2}_1= 1 , \\
A^{2}_2= \nu^{2}_2 ,
\end{cases}
&&
\begin{cases}
A^{3}_0 =0 , \\
A^{3}_{1} = 0 , \\
A^{3}_{2} = 1,
\end{cases}
\begin{cases}
B^{1}_0=1 , \\
B^{1}_1= \xi_1 ,
\end{cases}
&&
\begin{cases}
B^{2}_0=0 , \\
B^{2}_1= 1 ,
\end{cases}
\end{align*}
with $\nu^{1}_1$, $\nu^{1}_2$, $\nu^{2}_2$, and $\xi_1$ arbitrary constants.
We also define the initial condition matrices
\begin{align*}
\nu & \coloneqq
\left[\begin{matrix}
1& 0 & 0 \\
\nu^{1}_1 & 1 & 0 \\[.15cm]
\nu^{1}_{2} & \nu^{2}_{2} & 1 
\end{matrix} \right] ,
&
\xi&\coloneqq
\left[ \begin{matrix}
1& 0 \\
\xi^{1}_1 & 1
\end{matrix} \right] .
\end{align*}
In~\cite{espectral} the authors proved that the moments of $\d \Psi$ are defined in terms of the powers of $T$ by
\begin{align*}
\xi^{-1} E_{[2]}T^n E_{[3]}^\top \nu^{-\top} &=\int_\Delta x^n\d\psi (x),
&& n \in \mathbb N,
\end{align*}
which implies that the spectral measure of $T$ (Stieltjes--Markov type functions associated with $\Psi$) is given~by
\begin{align*}
\xi^{-1} E_{[2]} (z I- T)^{-1}E_{[3]}^\top \nu^{-\top}
&=\int_\Delta\frac{\d\psi (x)}{z-x}\eqqcolon\hat \psi (z) ,
\end{align*}
where
\begin{align*}
E_{[r]}\coloneqq
\left[
\begin{NiceArray}{cccc|cccccc}
1&0&\Cdots&0 & 0&\Cdots&&& &\\
0&\Ddots^{\text{$r$ times}}&\Ddots& \Vdots &\Vdots^{\text{$r$ times}}&&&&\\
\Vdots&\Ddots&\Ddots& 0 &&&&&\\
0&\Cdots&0& 1 &0&\Cdots&&&
\end{NiceArray}
\right]
 .
\end{align*}

Now, we reinterpret the mixed orthogonality with respect to the spectral measure.

\begin{thm}
\label{pro:biorthogonalitystieltjes}
Let $a$ and $b$ be the end points of $\Delta$, respectively.
Let $C$ be a circle, negatively oriented (clockwise), such that $a$ and $b$ are in the interior of~$C$.
Then the spectral measure of $ T $, $\hat\Psi$ is a complex measure of biorthogonality for $\big\{ \mathscr B_n \big\}_{n\in\mathbb N }$ and $\big\{\mathscr A_n\big\}_{n\in\mathbb N }$ over~$C$, i.e.
\begin{align}
\label{eq:ortogonalidadeS}
\int_C \mathscr B_n (z) \hat\Psi (z) \mathscr A_m^\top (z)\, \frac{\operatorname d z}{2\pi \operatorname i}
 = \delta_{n,m} , && n , m \in \mathbb N .
\end{align}
\end{thm}

\begin{proof}
We have the following identities
\begin{align*}
&
\int_C
\mathscr B_n (z) \hat \Psi (z) \mathscr A_m^\top (z)\, \frac{\operatorname d z}{2\pi \operatorname i} 
=\int_C \mathscr B_n (z)
\left(\int_\Delta \frac{\Psi (t)}{t-z}\, \operatorname d t \right) \mathscr A_m^\top (z)\, \frac{\operatorname d z}{2\pi \operatorname i} \\
&\phantom{olaola}
= \int_\Delta \left( \int_C \frac{\mathscr B_n (z)\Psi (t) \mathscr A_m^\top (z)}{t-z}\, \frac{\operatorname d z}{2\pi \operatorname i} \right)\, \operatorname d t \hspace{1.25cm} \text{(Fubini's theorem)}\\
&\phantom{olaola} = \int_\Delta \mathscr B_n (t) \Psi (t) \mathscr A_m^\top (t)\, \operatorname d t\hspace{1.625cm} \text{(Cauchy's integral theorem)} 
\end{align*}
and so from~\eqref{eq:biorthogonality} we get the desired result.
\end{proof}

\subsection{Gauss--Borel factorization}

Let us give a nice interpretation of this biorthogonality.
First of all, we can see that the Stieltjes--Markov like transformation of the weight matrix~$ \Psi$ is a generating function of the moments of $\Psi$. In fact,
\begin{align*}
\hat \Psi (z) = \int_{\Delta}\frac{ \psi (t)}{t-z}\operatorname d t
 = \sum_{n=0}^\infty \frac{ \int_{\Delta} t^n \Psi (t) \, \operatorname d t }{z^{n+1}}
 = \sum_{n=0}^\infty \frac{\Psi_{n}}{z^{n+1}} , 
\end{align*}
for $|z| > r \coloneqq \max \big\{ |t|, t \in \Delta \big\} $.
Hence, $\hat \Psi$ is an analytic function on compact sets of $\mathbb C \setminus \big\{ z \in \mathbb C : |z| < r \big\}$.
We write down the biorthogonality conditions~\eqref{eq:ortogonalidadeS} that we have just derived in Theorem~\ref{pro:biorthogonalitystieltjes},
\begin{align*}
\int_C \mathscr B_n (z) \hat \Psi (z) \mathscr A_m^\top (z)\, \frac{\operatorname d z}{2\pi \operatorname i}
 =
 \sum_{n=0}^\infty \int_C \mathscr B_n (z)\frac{\Psi_{n}}{z^{n+1}} \mathscr A_m^\top (z)\, \frac{\operatorname d z}{2\pi \operatorname i} ;
\end{align*}
and by the Cauchy integral formula we get
\begin{align}
\label{eq:bior}
 \delta_{n,m} = \int_C \mathscr B_n (z) \hat\Psi (z)\mathscr A_m^\top (z)\, \frac{\operatorname d z}{2\pi \operatorname i}
 =
 \sum_{k=0}^\infty \frac{\Big( \mathscr B_n (z) \Psi_{k} \mathscr A_m^\top (z) \Big)^{(k)}\Big|_{z \to 0}}{k!} .
\end{align}
Now, from the Leibniz rule for the derivatives, we know that
\begin{align*}
\frac1 {k!} \big( \mathscr B_n (z) \Psi_{k} \mathscr A_m^\top (z) \big)^{(k)}\Big|_{z \to 0} 
 =\sum_{j=0}^k \frac{ \big( \mathscr B_n \big)^{(j)} (0)}{j!} \Psi_k \frac{ \big( \mathscr A_m^\top \big)^{(k-j)} (0)}{(k-j)!} .
\end{align*}
With this identity we can reinterpret~\eqref{eq:bior} in matrix notation
\begin{align*}
\pmb{\operatorname B} \, \mathcal M \, \pmb{\operatorname A} = I
&& \text{where} &&
 \mathcal M &&
\text{is the matrix of moments~\eqref{eq:matriz_momentos},}
\end{align*}
and the matrices of Taylor polynomials coefficients of $\mathscr B_n$ and $\mathscr A_n$ are,
respectively,
\begin{align*} 
\pmb{\operatorname B}
  &=
 \left.\left[ \begin{smallmatrix}
 \mathscr B_0 
 & \\
 \mathscr B_1 
 & \big( \mathscr B_1 \big)^\prime 
 & \\
 \mathscr B_2 
 & \big( \mathscr B_2^\mathsf L\big)^\prime 
 & \frac 1{2!}\big( \mathscr B_2 \big)^{\prime\prime} 
 \\
 \vdots & \vdots & \vdots & \ddots \\
 \mathscr B_n 
 & \big( \mathscr B_n \big)^\prime 
 & \frac 1{2!}\big( \mathscr B_n \big)^{\prime\prime} 
 & \cdots & & \frac 1{n!}\big( \mathscr B_n \big)^{(n)} 
 \\
 \vdots & \vdots & \vdots & & & \vdots & \ddots
 \end{smallmatrix} \right] \right|_{z \to 0} ,
 \\
\pmb{\operatorname A}
  & =
 \left.\left[\begin{smallmatrix}
 \mathscr A_0^\top 
 & \\
 \mathscr A_1^\top 
 & \big( \mathscr A_1^\top \big)^\prime 
 & \\
 \mathscr A_2^\top 
 & \big( \mathscr A_2^\top \big)^\prime 
 & \frac 1{2!}\big( \mathscr A_2^\top \big)^{\prime\prime} 
 \\
 \vdots & \vdots & \vdots & \ddots \\
 \mathscr A_n^\top 
 & \big( \mathscr A_n^\top \big)^\prime 
 & \frac 1{2!}\big( \mathscr A_n^\top \big)^{\prime\prime} 
 & \cdots & & \frac 1{n!}\big( \mathscr A_n^\top \big)^{(n)} 
 \\
 \vdots & \vdots & \vdots & & & \vdots & \ddots
 \end{smallmatrix} \right]\right|_{z \to 0}^\top .
\end{align*}
As a conclusion, we get the Gauss--Borel factorization of the moment matrix
\begin{align}
\label{eq:gaussborel}
 \mathcal M =
\pmb{\operatorname B}^{-1} \, \pmb{\operatorname A}^{-1} .
\end{align}
We can see from~\eqref{eq:gaussborel} that the orthogonality relays on the Gauss--Borel factorization of the moment matrix.
Remember (cf.~\cite{espectral}) that the necessary and sufficient conditions in order to assure Gauss--Borel factorization of $\mathcal M$ is that~\eqref{eq:regular} takes place.

\subsection{Characteristic polynomials}

For the semi-infinite matrix $T$ in~\eqref{eq:monic_Hessenberg} we consider the polynomials~$P_N(x)$ as the characteristic polynomials of the truncated matrices~$T^{[N-1]}$,~i.e.,
\begin{align*}
P_{0}(x) = 1,
 &&
P_{N}(x)&
\coloneqq \det\big(xI_N-T^{[N-1]}\big), && N\in\mathbb Z_+.
\end{align*}
Obviously, $\deg P_N=N$.
For Hessenberg matrices~\cite{PBF_1} it happens that the characteristic polynomials up to a factor coincides with the right recursion polynomials. However, for the banded situation this does not hold in general. Nevertheless, the relation between determinants of the recursion polynomials, right or left, with the characteristic polynomials of the banded matrix $T$ relays
 on the following matrices of left and right recursion polynomials
\begin{align*}
A_N &\coloneqq
\left[
\begin{matrix}
A^{1}_N& A^{1}_{N+1} & A^{1}_{N+2} \\[2pt]
A^{2}_N& A^{2}_{N+1} & A^{2}_{N+2} \\[2pt]
A^{3}_N& A^{3}_{N+1} & A^{3}_{N+2}
\end{matrix}
\right], &
 B_N &\coloneqq
\left[
\begin{matrix}
B^{1}_N & B^{2}_N \\[2pt]
B^{1}_{N+1} & B^{2}_{N+1}
\end{matrix}
\right],& N&\in \mathbb N ,
\end{align*}
and the following products, defining for all $N \in \mathbb Z_+$,
\begin{align*}
\alpha_N &\coloneqq T_{3,0} \cdots T_{N+2,N-1}, &
\beta_N &\coloneqq (-1)^{N}T_{0,2}\cdots T_{N-1,N+1} ,
\end{align*}
with $\alpha_0=\beta_0=1$.
In fact, in~\cite{espectral}, it was proved that
the characteristic polynomials and determinants of left and right recursion polynomial blocks~satisfy
\begin{align*}
P_N(x)&=\alpha_N \det A_N(x)=\beta_N \det B_N(x), && N \in \mathbb N .
\end{align*}
This fact leads us directly to the right and left eigenvectors of $T^{[N]}$.
Let us introduce the determinantal polynomials
\begin{align}\label{eq:QNn}
Q_{n,N}&\coloneqq
\left|
\begin{matrix}
A^{1}_{n} & A^{2}_{n} & A^{3}_{n} \\[2pt]
A^{1}_{N+1} & A^{2}_{N+1} & A^{3}_{N+1} \\[2pt]
A^{1}_{N+2} & A^{2}_{N+2} & A^{3}_{N+2}
\end{matrix}
\right|,&
R_{n,N}&\coloneqq
\left|\begin{matrix}
B^{1}_{n} & B^{2}_{n} \\[2pt]
B^{1}_{N+1} & B^{2}_{N+1}
\end{matrix}
\right|,
&& n, N \in \mathbb N ,
\end{align}
and the semi-infinite row and column vectors
\begin{align*}
Q_N&\coloneqq
\left[\begin{matrix}
Q_{0,N} &Q_{1,N} &\cdots
\end{matrix}\right], &
R_N&\coloneqq
\left[\begin{matrix}
R_{0,N} & R_{1,N} & \cdots
\end{matrix}\right]^\top .
\end{align*}
With the corresponding truncations
\begin{align*}
Q^{\langle N\rangle}&\coloneqq
 \left[\begin{matrix}
Q_{0,N} &Q_{1,N}&\cdots & Q_{N,N}
\end{matrix} \right], 
& R^{\langle N\rangle}&
\coloneqq
\left[\begin{matrix}
R_{0,N} & R_{1,N} & \cdots & R_{N,N}
\end{matrix}\right]^\top
 ,
\end{align*}
we define the left and right eigenvectors for $T^{[N]}$. In fact, it can be proven that (cf.~\cite{espectral}), 
\begin{align*}
Q^{\langle N\rangle}\big|_{x=\lambda^{[N]}_k}
&& \text{and} &&
R^{\langle N\rangle}\big|_{x=\lambda^{[N]}_k},
&& k\in\{1,\dots,N+1\},
\end{align*}
are left and right eigenvectors of $T^{[N]}$, respectively.


We present now a generalized Christoffel--Darboux formula for the determinantal polynomials and the characteristic polynomial of a banded matrix (cf.~\cite{espectral}).
For the determinantal polynomials $Q_{n,N}$ and $R_{n,N}$ introduced in~\eqref{eq:QNn} we get the following generalized Christoffel--Darboux formula
\begin{align*}
\sum_{n=0}^{N}Q_{n,N}(x)R_{n,N}(y)= \frac{1}{\alpha_N \beta_N }\frac{P_{N+1}(x)P_{N}(y)-P_{N}(x)P_{N+1}(y)}{x-y} ,
 && N \in \mathbb N ,
\end{align*}
and its confluent Christoffel--Darboux relation, for all $ N \in \mathbb N$,
\begin{align*}
\sum_{n=0}^{N}Q_{n,N} (x) R_{n,N} (x) = \frac{1}{\alpha_N \beta_N }\big(P^\prime_{N+1} (x) P_{N} (x)-P^\prime_{N} (x) P_{N+1} (x)\big) .
\end{align*}
We end this subsection with some important interlacing result (cf.~\cite{espectral}).

\begin{thm}[Interlacing]\label{pro:interlacing}
Let us assume that $T$ is oscillatory. Then:
\begin{enumerate}[\rm (1)]
\item The polynomial $P_{N+1}$ interlaces $P_N$.
\item For $x\in\mathbb R$, for the corresponding Wronskian we find $P_{N+1}^\prime P_N-P_N^\prime P_{N+1}>0$. 
In~particular,
\begin{align*}
(P_{N+1}^\prime P_N)\big|_{x=\lambda^{[N]}_k}&>0, & (P_{N+1}P_N^\prime )\big|_{x=\lambda^{[N-1]}_k}&<0.
\end{align*}
\item The confluent kernel is a positive function; i.e.,
\begin{align*}
\alpha_N \beta_N \sum_{n=0}^{N}Q_{n,N}(x)R_{n,N}(x)>0 && \text{for} && x\in\mathbb R .
\end{align*}
\end{enumerate}
\end{thm}

\subsection{Biorthogonality and Christoffel numbers}

We now discuss, for the truncated situation, how to construct biorthogonal families of left and right eigenvectors and introduce the Christoffel numbers in this setting.
Definig the Christoffel numbers~as
%
%
\begin{align*}
\mu_{k,1}^{[N]} & \coloneqq
\frac{\alpha_N
\left|\begin{smallmatrix}
A^{2}_{N+1} (\lambda^{[N]}_k) & A^{3}_{N+1}(\lambda^{[N]}_k) \\
A^{2}_{N+2} (\lambda^{[N]}_k) & A^{3}_{N+2}(\lambda^{[N]}_k) 
\end{smallmatrix}
\right|}{P^\prime_{N+1} (\lambda^{[N]}_k) P_{N} (\lambda^{[N]}_k)},
 &&
\mu^{[N]}_{k,2} \coloneqq
 - \frac{\alpha_N\left|\begin{smallmatrix}
A^{1}_{N+1} (\lambda^{[N]}_k )& A^{3}_{N+1} (\lambda^{[N]}_k ) \\[2pt]
A^{1}_{N+2} (\lambda^{[N]}_k )& A^{3}_{N+2} (\lambda^{[N]}_k )
\end{smallmatrix}
\right|}{P^\prime_{N+1} (\lambda^{[N]}_k) P_{N} (\lambda^{[N]}_k)},
 \\
\mu^{[N]}_{k,3} & \coloneqq
\frac{\alpha_N\left|\begin{smallmatrix}
A^{1}_{N+1} (\lambda^{[N]}_k ) & A^{2}_{N+1} (\lambda^{[N]}_k ) \\[2pt]
A^{1}_{N+2} (\lambda^{[N]}_k ) & A^{2}_{N+2} (\lambda^{[N]}_k )
\end{smallmatrix}
\right|}{P^\prime_{N+1} (\lambda^{[N]}_k) P_{N} (\lambda^{[N]}_k)},
 \\
\rho_{k,1}^{[N]}&\coloneqq \beta_N
B^{2}_{N+1} (\lambda^{[N]}_k ) , &&
\rho^{[N]}_{k,2} \coloneqq
 -\beta_N
B^{1}_{N+1} (\lambda^{[N]}_k )
 ,
\end{align*}
we get the following orthogonality relations.

\begin{thm}[Spectral properties]\label{pro:UW}
Assume that $P_{N+1}$ has simple zeros at the set $\big\{\lambda^{[N]}_k, k =1 , \ldots , N+ 1 \big\}$,
then
\begin{enumerate}[\rm (I)]

\item The following expression holds
\begin{align*}
w^{\langle N\rangle}_{k,n}&=
\frac{ \alpha_N Q_{n-1,N} (\lambda^{[N]}_k)
}{
P_{N}(\lambda^{[N]}_k)P^\prime_{N+1}(\lambda^{[N]}_k)}, &
 u^{\langle N\rangle}_{k,n}&=\beta_N R_{n-1,N} (\lambda^{[N]}_k ).
\end{align*}
\item For the Christoffel numbers it holds that
\begin{align}
\label{eq:kcomponentelefteigenII}
\left[
\begin{matrix}
\mu^{[N]}_{k,1} \\[.1cm]
\mu^{[N]}_{k,2} \\[.1cm]
\mu^{[N]}_{k,3}
\end{matrix}
\right]
&= \left[
\begin{matrix}
1& 0 & 0 \\[.1cm]
\nu^{1}_1 & 1 & 0 \\[.1cm]
\nu^{1}_{2} & \nu^{2}_{2} & 1 
\end{matrix}
\right]^{-1}
\left[
\begin{matrix}
w^{\langle N\rangle}_{k,1} \\[.1cm]
w^{\langle N\rangle}_{k,2} \\[.1cm]
w^{\langle N\rangle}_{k,3}
\end{matrix}
 \right], &
\left[
\begin{matrix}
\rho^{[N]}_{k,1} \\[.1cm]
\rho^{[N]}_{k,2}
\end{matrix}
\right]
&= \left[
\begin{matrix}
1& 0 \\[.1cm]
\xi^{1}_1 & 1
\end{matrix}
\right]^{-1}
\left[
\begin{matrix}
u^{\langle N\rangle}_{k,1} \\[.1cm]
u^{\langle N\rangle}_{k,2}
\end{matrix}
 \right]
 .
\end{align}
 
\item The corresponding matrices $\mathscr U$ (with columns the right eigenvectors $u_k$ arranged in the standard order) and 
$\mathscr W$ (with rows the left eigenvectors~$w_k$ arranged in the standard order) satisfy
\begin{align}\label{eq:UW=I}
\mathscr U\mathscr W=\mathscr W\mathscr U=I_{N+1}.
\end{align}
\item In terms of the eigenvalues diagonal matrix $D=\operatorname{diag}
\left[\lambda^{[N]}_1,\ldots,\lambda^{[N]}_{N+1}\right]$ we have
\begin{align}\label{eq:UDnW=Jn}
\mathscr UD^n\mathscr W&=\big(T^{[N]}\big)^n, & n& \in\mathbb N .
\end{align}
\end{enumerate}
\end{thm}

\section{Spectral Theorem} \label{sec:4}

As in the scalar case we will reformulate the previous biorthogonality relations in terms of a set of discrete measures and corresponding mixed multiple discrete orthogonality.
Let us consider the following step functions for $a \in \{ 1,2,3 \}$ and $b \in \{ 1,2\}$,
\begin{align}
\label{eq:conv_entrada}
\psi^{[N]}_{b,a}&
=
 \begin{cases}
0, & x<\lambda^{[N]}_{N+1},\\[2pt]
\rho^{[N]}_{1,b}\mu^{[N]}_{1,a}+\cdots+\rho^{[N]}_{k,b}\mu^{[N]}_{k,a},
&
 \lambda^{[N]}_{k+1}\leqslant x< \lambda^{[N]}_{k},
\quad k\in\{1,\dots,N\}, \\[2pt]
\rho^{[N]}_{1,b}\mu^{[N]}_{1,a}+\cdots+\rho^{[N]}_{N+1,b}\mu^{[N]}_{N+1,a} ,
& x \geqslant \lambda^{[N]}_{1}.
\end{cases}
\end{align}
It can be proven that for $a \in \{ 1,2,3 \}$ and $b \in \{ 1,2\}$,
\begin{align}\label{eq:bound}
\rho^{[N]}_{1,b}\mu^{[N]}_{1,a}+\cdots+\rho^{[N]}_{N+1,b}\mu^{[N]}_{N+1,a} = (\xi^{-1}I_{2,3}\nu^{-\top})_{b,a} ,
\end{align}
and so in the case of positive Christoffel coefficients these step functions are non-decreasing and uniformly bounded in~$N$.
Notice that these functions have bounded variation and are right continuous, so it makes sense to consider the  Lebesgue--Stielt\-jes associated measures. Let us introduce a $2\times 3$~matrix 
\begin{align*}
\Psi^{[N]}\coloneqq
\left[\begin{matrix}
\psi^{[N]}_{1,1}&\psi^{[N]}_{1,2} &\psi^{[N]}_{1,3}\\[.1cm]
\psi^{[N]}_{2,1}&\psi^{[N]}_{2,2} &\psi^{[N]}_{2,3}
\end{matrix}
\right]
 &&
 \text{and}
 &&
\d\Psi^{[N]}
=
\left[\begin{matrix}
\d\psi ^{[N]}_{1,1}&\d\psi ^{[N]}_{1,2} &\d\psi^{[N]}_{1,3}\\[.1cm]
\d\psi ^{[N]}_{2,1}&\d\psi ^{[N]}_{2,2} &\d\psi^{[N]}_{2,3}
\end{matrix}\right],
\end{align*}
the discrete Lebesgue--Stieltjes measures supported at the zeros of $P_{N+1}$, by
\begin{align*}
\d\Psi^{[N]}
=\sum_{k=1}^{N+1}
\left[\begin{matrix}
\rho^{[N]}_{k,1} \\[.1cm]
\rho^{[N]}_{k,2}
\end{matrix}\right]
\left[\begin{matrix}
\mu^{[N]}_{k,1} & \mu^{[N]}_{k,2} & \mu^{[N]}_{k,3}
\end{matrix}
\right]
\, \delta (x-\lambda^{[N]}_k ).
\end{align*}
Now, from~\eqref{eq:UW=I},
the following biorthogonal relations hold
\begin{align}\label{eq:N_biorthogonality}
\int \mathscr B_n (x) \d \psi (x) \mathscr A_m^\top (x) = \delta_{n,m}, && n , m \in \{0,1, \ldots, N \} .
\end{align}
Defining the second kind characteristic polynomial matrix as
\begin{align*}
P^{(1)}_{N+1} (z)
& \coloneqq \xi^{-1} E_{[2]}\operatorname{adj}(z I_{N+1}-T^{[N]}) E_{[3]}^\top\nu^{-\top} \\
 & =
\sum_{k=1}^{N+1} 
\begin{bNiceMatrix}
\rho^{[N]}_{k,1} \\[.1cm] \rho^{[N]}_{k,2}
\end{bNiceMatrix}\begin{bNiceMatrix}
\mu^{[N]}_{k,1} &\mu^{[N]}_{k,2} & \mu^{[N]}_{k,3}
\end{bNiceMatrix}
\prod_{\mathclap{\substack{l\in\{1,\dots,N+1\}\\l\neq k}}}\big(x-\lambda^{[N]}_l\big)
\\
&=\int \frac{P_{N+1}(z)-P_{N+1}(x)}{z-x}\d\Psi^{[N]}(x),
\end{align*}
we get a $2\times 3$ matrix of polynomials whose entries are the polynomials of the second kind:
$\big( P^{(1)}_{N+1} \big)_{b,a}=P^{(b,a)}_{N+1} $, 
$ a \in \{1,2,3\} $, 
$b \in \{ 1,2 \} $.
The resolvent matrix $ R^{[N]}(z)$ of the leading principal submatrix $ T^{[N]}$ is
\begin{align*}
R^{[N]}(z)\coloneqq \big(z I_{N+1}- T^{[N]}\big)^{-1}=\frac{\operatorname{adj}\big(z I_{N+1}- T^{[N]}\big)}{\det(z I_{N+1}-T^{[N]})} ,
\end{align*}
and by the spectral decomposition of the matrix $ T^{[N]}$~\eqref{eq:UDnW=Jn} we get
\begin{align*}
R^{[N]}(z) & =\mathscr U(zI_{N+1}-D)^{-1}\mathscr W. 
\end{align*}
For the $2\times 3$ matrix of Weyl functions
\begin{align}
\label{eq:convweyl}
S^{[N]}=
\left[\begin{matrix}
S^{[N]}_{1,1} &S^{[N]}_{1,2}& S^{[N]}_{1,3} \\[.1cm]
S^{[N]}_{2,1} &S^{[N]}_{2,2}& S^{[N]}_{2,3} 
\end{matrix}
\right]
\coloneqq \xi^{-1}E_{[2]}R^{[N]} (z)E^\top_{[3]}\nu^{-1} ,
\end{align}
we can write
\begin{align*}
S^{[N]}(z)&= \frac{P^{(1)}_{N+1}(z)}{P_{N+1}(z)}
=\sum_{k=1}^{N+1}\frac{1}{z-\lambda^{[N]}_k}\begin{bNiceMatrix}
\rho^{[N]}_{k,1} \\[.1cm]\rho^{[N]}_{k,2}
\end{bNiceMatrix}\begin{bNiceMatrix}
\mu^{[N]}_{k,1} &\mu^{[N]}_{k,2} & \mu^{[N]}_{k,3}
\end{bNiceMatrix}
  = \int\frac{\d\Psi^{[N]}(x)}{z-x}.
\end{align*}
We say that a banded matrix $ T$ as in \eqref{eq:monic_Hessenberg} admits a positive bidiagonal factorization~(PBF)~if
\begin{align*}
T= L_{1} L_2 L_{3} \Delta U_2 U_1,
\end{align*}
with $\Delta=\operatorname{diag} \left[\Delta_0,\Delta_1,\ldots \right]$ and bidiagonal matrices given respectively by 
\begin{align*}
L_k&\coloneqq \left[\begin{NiceMatrix}[columns-width=auto]
1 &0&\Cdots&\\
 L_{k|0} & 1 &\Ddots&\\
0& L_{k|1}& 1& \\
\Vdots&\Ddots& \Ddots& \Ddots\\&&&
\end{NiceMatrix}\right], & 
U_j& \coloneqq
\left[\begin{NiceMatrix}[columns-width = auto]
1& U_{j|0}&0&\Cdots&\\
0& 1& U_{j|1}&\Ddots&\\
\Vdots&\Ddots&1&\Ddots&\\
& &\Ddots &\Ddots &\\&&&&
\end{NiceMatrix}\right], 
\end{align*}
and such that the positivity constraints $L_{k|i}, U_{j|i}, \Delta_i>0,$ for $i\in\mathbb N$, $k \in \{ 1,2,3 \}$, $j \in \{ 1,2 \}$, are satisfied. 
Now, we consider the Darboux transformations of our truncations,~$T^{[N]}$,~i.e.
\begin{align}\label{eq:Darboux+}
\begin{cases}
\hat T^{[N,-1]}=U_1^{[N]}L_1 ^{[N]} L_2 ^{[N]} L_3 ^{[N]} \Delta^{[N]} U_2^{[N]} , \\
\hat T^{[N,-2]}=U_2^{[N]}U_1^{[N]}L_1^{[N]} L_2^{[N]} L_3 ^{[N]} \Delta^{[N]} , \\
\hat T^{[N,+1]}=L_2 ^{[N]} L_3 ^{[N]} \Delta^{[N]} U_2^{[N]} U_1^{[N]} L_1^{[N]}, \\
\hat T^{[N,+2]}=L_3 ^{[N]} \Delta^{[N]} U_2^{[N]} U_1^{[N]}L_1^{[N]}L_2 ^{[N]}, \\
\hat T^{[N,+3]}= \Delta^{[N]} U_2^{[N]} U_1^{[N]}L_1^{[N]}L_2 ^{[N]} L_{3} ^{[N]} .
\end{cases}
\end{align}
We can prove the following fundamental result.

\begin{thm}[cf.~\cite{espectral}]
\label{lemma:poly}
Let us assume that $T$ has a PBF.
Then, for $k\in\{1,2,3\}$:
\begin{enumerate}[\rm (1)]
\item The Darboux transformations $\hat{ T}^{[N,+a]}$, $a\in\{1,2,3\}$, $\hat{ T}^{[N,-b]}$, $b\in\{1,2\}$ are oscillatory.
\item The characteristic polynomial of the Darboux transformations 
$\hat{ T}^{[N,+a]}$, $a\in\{1,2, 3\}$, $\hat{ T}^{[N,-b]}$, $b\in\{1,2\}$ is $P_{N+1}$. 
\item If $w,u$ are left and right eigenvectors of $ T^{[N]}$, respectively, then
$\hat w_1=w L_{1}^{[N]} L_{2}^{[N]}$,
$\hat w_2=w L_{1}^{[N]} $, 
$\hat w_3=w L_{1}^{[N]} L_{2}^{[N]} L_{3}^{[N]}$, 
are left eigenvectors of $\hat{ T}^{[N,+2]}$, $\hat{ T}^{[N,+1]}$, $\hat{ T}^{[N,+3]}$, respectively
 and
$\hat u_1=U_{2}^{[N]} U_{1}^{[N]}u$,
$\hat u_2= U_{1}^{[N]}u$
are right eigenvector of $\hat{ T}^{[N,-2]}$ and $\hat{ T}^{[N,-1]}$, respectively.
\end{enumerate}
\end{thm}

Now, for the proof of the positivity of the Christoffel numbers it is important to recall the representation of the Christoffel numbers~\eqref{eq:kcomponentelefteigenII}, apply the Theorem~\ref{pro:interlacing} and use the oscillatory character of $\hat{ T}^{[N,j]}$, $j \in \{ -2,-1,1,2,3 \}$ (cf.~\cite{espectral}).

\begin{thm}[Normal convergence of Weyl functions]
If the banded matrix~$T$ is bounded and there exist $s\geqslant 0$ such that $T+sI$ has a PBF, then the Weyl functions~\eqref{eq:convweyl} converge uniformly in compact subsets of $\bar{\mathbb C}\setminus \Delta$ to the Stieltjes--Markov functions, i.e.,
\begin{align*}
\lim_{N\to\infty}
S^{[N]}
 (z)
=
\lim_{N\to\infty}
 \frac{P^{(1)}
 _{N+1}(z)}{P_{N+1}(z)}
 =
 \hat \psi
 (z) && \text{uniformly on} && \bar{\mathbb C}\setminus \Delta.
\end{align*}
\end{thm}

If $T+ s I$ has a PBF then its leading principal submatrices $T^{[N]} $ are oscillatory.
The shift in the matrix $T\to T+sI$ only shifts by $s$ the eigenvalues of the truncations $T^{[N]}$, so that they are positive, and the dependent variable of the recursion polynomials, but do not alter the interlacing properties of the polynomials and the positivity of the corresponding Christoffel numbers (cf.~for instance~\cite{espectral}). 
Now, the sequences $\big\{\psi_{a,b}^{[N]}\big\}_{N=0}^\infty$, $a\in\{1,2,3\}$, $b\in\{1,2\}$ given in~\eqref{eq:conv_entrada} are positive.
Moreover,~\eqref{eq:bound} implies that they are uniformly bounded and nondecreasing.
Consequently, following Helly's results, see~\cite[\S II]{Chihara} there exist subsequences that converge when $N\to\infty$ to positive nondecreasing functions $\psi_{b,a}$ with support on $\Delta$ and that the discrete biorthogonal relations lead to the stated biorthogonal properties.

\medskip

As a nice application of the work just presented we have the Gaussian quadrature formulas in the theory of  mixed-multiple orthogonal polynomials.

\begin{thm}
The following Gauss quadrature formulas hold
\begin{align*}
\int_\Delta x^n\d\psi_{b,a}(x)&=\sum_{k=1}^{N+1}\rho_{k,b}^{[N]}\mu_{k,a}^{[N]} ( \lambda_k^{[N]} )^n, & 0&\leqslant n\leqslant d_{b,a}(N) ,
\end{align*}
where $ a \in\{1,2,3\}$, $b\in\{1,2\}$ and
\begin{align*}
d_{b,a}(N)
\coloneqq
\left\lceil\frac{ N+2-a}{3}\right\rceil+\left\lceil\frac{ N+2-b}{2}\right\rceil-1 .
\end{align*}
\end{thm}

Here the degrees of precision $d_{b,a}$ are optimal.

%
%

\end{document}